\documentclass[10pt,a4paper]{amsart}

\usepackage{latexsym, amsthm}

\usepackage{amsfonts, amsmath, amssymb}

\usepackage{hyperref}

\usepackage{euscript,mathrsfs}

\usepackage{textcomp}

\newtheorem{theorem}{Theorem}[section]
\newtheorem{lemma}[theorem]{Lemma}
\newtheorem{proposition}[theorem]{Proposition}

\newtheorem{corollary}[theorem]{Corollary}

\theoremstyle{definition}
\newtheorem{definition}[theorem]{Definition}
\newtheorem{example}[theorem]{Example}

\theoremstyle{remark}
\newtheorem{remark}[theorem]{Remark}

\def\Fq{{\mathbb F}_q}

\def\I{{\mathcal I}}
\def\AA{{\mathbb A}}
\def\FF{{\mathbb F}}

\def\PP{{\mathbb P}}
\def\M{\mathcal{M}}
\def\N{{\mathbb N}}

\def\J{\mathcal{J}}

\def\Q{{\mathbb Q}}
\def\ZZ{{\mathbb Z}}

\def\hp3{\widehat{\mathbb P}^3}

\def\Aff{{\mathbb A}}
\def\Z{\mathbb{Z}}
\def\lm{\mathsf{lm}}
\newcommand{\Zs}{{\mathscr{Z}}}
\newcommand{\p}{\mathit{p}}
\newcommand{\Mon}{\mathbb{M}}
\newcommand{\Monb}{\overline{\mathbb{M}}}
\newcommand{\LT}{\mathrm{LT}}
\newcommand{\supp}{\mathrm{supp}}

\def\N{{\mathbb N}}
\def\ZZ{{\mathbb Z}}

\def\X{\mathscr{X}}
\def\Gs{\mathrm{\Gamma}_q}

\begin{document}
\title[Projective Footprint Bound]{Vanishing ideals of projective spaces over finite fields and a projective footprint bound}
\author{Peter Beelen}
\address{Department of Applied Mathematics and Computer Science, \newline \indent
Technical University of Denmark, DK 2800, Kgs. Lyngby, Denmark}
\email{pabe@dtu.dk}

\author{Mrinmoy Datta}
\address{Department of Applied Mathematics and Computer Science, \newline \indent
Technical University of Denmark, DK 2800, Kgs. Lyngby, Denmark}
\curraddr{Department of Mathematics and Statistics,  
UiT - The Arctic University of Norway, N-9037, Troms\o, Norway}
\email{mrinmoy.dat@gmail.com}

\author{Sudhir R. Ghorpade}
\address{Department of Mathematics,
Indian Institute of Technology Bombay,\newline \indent
Powai, Mumbai 400076, India.}
\email{srg@math.iitb.ac.in}


\date{}

\begin{abstract}
We consider the vanishing ideal of a projective space over a finite field. An explicit set of generators for this ideal has been given by Mercier and Rolland. 
We show that these generators form 
a universal Gr\"obner basis of the ideal. Further we give a projective analogue of the 
footprint bound, and a version of it that is suitable for estimating the number of rational points of projective algebraic varieties over finite fields. 
An application to Serre's inequality for the number of rational points of projective hypersurfaces over finite fields is included. \end{abstract} 



\maketitle

\section{Introduction}
Let $\Fq$ be the finite field with $q$ elements and let $\overline{\FF}_q$ denote an algebraic closure of $\Fq$. We are primarily interested in the problem of determining or estimating the number of $\Fq$-rational points of an affine or projective variety $\X$ defined over $\Fq$. When $\X$ is known to be irreducible (over $\overline{\FF}_q$) and better still, nonsingular, then there are good estimates that arise from deeper methods in algebraic geometry, including the known validity of the Weil conjectures and related results such as the Grothendieck-Lefschetz trace formula. These estimates typically involve topological data such as the genus (when $\X$ is a smooth projective curve) or more generally, the $\ell$-adic Betti numbers. The estimates are particularly good when $\X$ is a complete intersection.  Simplest among such general estimates is the Lang-Weil inequality, which still requires $\X$ to be (absolutely) irreducible with a given embedding in a projective space, and the knowledge of the dimension and degree of $\X$. One may refer to \cite{GL} for a brief survey of these aspects. 

What if $\X$ is not necessarily irreducible? A reducible $d$-dimensional  algebraic variety can have many more $\Fq$-rational points than any irreducible $d$-dimensional  algebraic variety, and the geometric estimates such as those mentioned above are not immediately applicable. A recent result of Couvreur \cite{C} does provide a way out, but still it requires the knowledge of dimensions as well as degrees of the irreducible components. A more elementary approach that relies on simpler attributes of varieties such as the number of defining equations and their degrees may be desirable for practical applications. The simplest among such results is the bound due to Ore (1933) for the number of $\Fq$-rational points of affine hypersurfaces of degree $d$ defined over $\Fq$.
It states that if 
$\Zs = Z(f)$  is 
the zero-set in $\AA^m(\overline{\FF}_q)$ of a nonzero polynomial 
$f\in \Fq[x_1, \dots , x_m]$ of degree $\le d$,~then 
$$
|\Zs(\Fq)| \le d q^{m-1}. 
$$
The bound $dq^{m-1}$ is a good bound in the sense that it is attained when $d\le q$. An analogous good bound for projective hypersurfaces was conjectured by Tsfasman and proved by Serre \cite{Se} and, independently,  S{\o}rensen \cite{soer} in 1991. It states that if $\X = V(F)$ is the zero-set in $\PP^m(\overline{\FF}_q)$ of a homogeneous polynomial 
$F\in \Fq[x_0, x_1, \dots , x_m]$ of degree $d$, then 
\begin{equation}
\label{SerreIneq}
|\X (\Fq)| \le dq^{m-1} + \p_{m-2} , 
\end{equation}
where for $j\in \ZZ$, by $\p_j$ we denote $|\PP^j(\Fq)|$, i.e., $\p_j:= q^j + q^{j-1} + \dots + q + 1$ if $j\ge 0$, while we set $\p_j:=0$ if $j< 0$.  

In the case of affine algebraic varieties, say $\Zs = Z(f_1, \dots , f_r)$ 
in $\AA^m(\overline{\FF}_q)$, where  $f_1, \dots , f_r \in \Fq[x_1, \dots , x_m]$, there is an attractive alternative for determining or estimating the number of $\Fq$-rational points of $\Zs$. Namely, instead of 
the high dimensional variety $\Zs$,  we consider the zero-dimensional variety ${\Zs}_q = Z(f_1, \dots , f_r, x_1^q - x_1, \dots, x_m^q - x_m)$ in $\AA^m(\overline{\FF}_q)$ and observe that $|\Zs(\Fq)| = |{\Zs}_q|$. In algebraic terms (and using the superscript `$a$' to indicate the affine setting), instead of the ideal $I^a = \langle f_1, \dots , f_r \rangle$ of the polynomial ring $\overline{\FF}_q[x_1, \dots , x_m]$, we consider the larger ideal $I_q^a := I^a + \Gamma^a_q$, where $\Gamma_q^a:= \langle x_1^q - x_1, \dots, x_m^q - x_m \rangle$ is the vanishing ideal of $\AA^m(\Fq)$. Now it is a general fact 
that if $k$ is an arbitrary algebraically closed field and $J^a$ is an ideal of $k[x_1, \dots , x_m]$ whose 
zero-set $Z(J^a)$ in $\AA^m(k)$ is finite, then 
\begin{equation}
\label{FBound}
|Z(J^a)| \le  | \Delta(J^a) | \quad \text{with equality if  $J^a$ is a radical ideal}, 
\end{equation} 
where $\Delta(J^a) $ is the 
\emph{footprint} of $J^a$, defined (with respect to a fixed monomial order on the set $\Mon^a$ of all monomials in $x_1, \dots , x_m$) by 
$$
\Delta(J^a) := \{ \mu \in \Mon^a : \mu\ne \lm (f) \text{ for any } f \in J^a \text{ with } f\ne 0\} ,
$$
where for any nonzero polynomial $f \in k[x_1, \dots , x_m]$, by $\lm (f)$ we denote the \emph{leading monomial} of $f$ (with respect to the fixed monomial order). 
It is customary to refer to \eqref{FBound} as the \emph{footprint bound}. A proof of \eqref{FBound},  
can be found, e.g., in \cite[Thm. 8.32]{BW}.

In case a Gr\"obner basis of $J^a$, say $\{f_1, \dots , f_s\}$, can be found, then clearly, $\Delta(J^a)=  \{ \mu \in \Mon^a : \lm (f_i) \nmid \mu \text{ for } i=1, \dots , s\}$. What makes the case of $k = \overline{\FF}_q$ and $J^a= I_q^a = I^a + \Gamma^a_q$ particularly nice is that the usual generators $ x_1^q - x_1, \dots, x_m^q - x_m$ form a Gr\"obner basis of $\Gamma^a_q$ (since their leading monomials 
are pairwise coprime; see, e.g., \cite[Ch. 2, \S 9]{CLO}), 
and regardless of what  the ideal  $I^a$ is, $I_q^a$ is always a radical ideal,
(and, in fact, the vanishing ideal of $Z(I^a)(\Fq)$), 
thanks to a classical result of Terjanian \cite{Terj}. Thus, in this case \eqref{FBound} specializes to the 
$\Fq$-\emph{footprint formula}: 
\begin{equation}
\label{FFormula}
|\Zs(\Fq)| = |\Delta(I^a_q)|, \quad \text{where} \quad \Zs: = Z(I^a).
\end{equation} 
In practice, it is more convenient to consider the well-known notion of the reduction $\overline{f}$ of a nonzero polynomial $f\in \overline{\FF}_q[x_1, \dots , x_m]$ (see, e.g., \cite[Ch. 2]{Joly} or \cite[\S\,II]{BGH}) and observe that $\Delta(I^a_q) = \overline{\Delta}(I^a)$, where for any subset $T^a$ of $\overline{\FF}_q[x_1, \dots , x_m]$, we define
$$
\overline{\Delta}(T^a):=\left\{ \mu \in \Mon^a : \mu = \overline{\mu} \text{ and } \lm(\overline{f}) \nmid \mu \text{ for all } f\in T^a \text{ with } \overline{f} \ne 0\right\}.
$$
In particular, if 
$\{ f_1, \dots , f_r\}$ is any set of generators of $I^a$,
then 
\begin{equation}
\label{AFBound}
|\Zs(\Fq)| \le |\overline{\Delta}(f_1, \dots , f_r)|,  \quad \text{where} \quad \Zs: = Z(I^a) = Z(f_1, \dots , f_r).
\end{equation}
We shall refer to \eqref{AFBound} as the \emph{affine $\Fq$-footprint bound}. Some 
early references for footprint bounds such as these, and especially their applications to coding theory, are the works of H{\o}holdt \cite{Ho}, Fitzgerald and Lax \cite{FL}, and Geil and H{\o}holdt \cite{GH}. 

One of the most general estimates known for the number of $\Fq$-rational points of affine algebraic varieties defined over $\Fq$ is a result of Heijnen and Pellikaan~\cite{HP} that gives, in fact, the maximum possible value $e_r^{\Aff}(d,m)$ of $|\Zs(\Fq)|$, where $\Zs = Z(f_1, \dots , f_r)$ 
and $\{f_1, \dots , f_r\}$ vary over sets of $r$ linearly independent polynomials of degree $\le d$ in  $\Fq[x_1, \dots , x_m]$. Recent works such as \cite{GM} and \cite{BD} have shown that it is possible to derive results like these, and even more general results, by a careful study of footprints and shadows\footnote{The notion of shadow is complementary to that of footprint.}, together with some nontrivial combinatorial results such as the Kruskal-Katona theorem and its variants. On the other hand, the corresponding projective problem of the determination of the maximum number $e_r(d,m)$ of 
$|\X(\Fq)|$, where $\X = V(F_1, \dots , F_r)$ and $\{F_1, \dots , F_r\}$ vary over sets of $r$ linearly independent homogeneous polynomials of degree $d$ in  $\Fq[x_0, x_1, \dots , x_m]$ is still open, in general, although there has been some recent 
progress 
(see, e.g., \cite{DG, DG1, BDG}). 

Motivated by the above considerations, we begin in this paper the investigation of the $\Fq$-rational points of projective varieties $\X = V(\I)$ by associating to $\X$ a zero-dimensional variety ${\X}_q$ and developing suitable analogues of the footprint bound. The first question to be asked is an explicit  determination of the ideal $\Gs = I(\PP^m(\Fq))$ and its Gr\"obner basis. 
The first part is answered by Mercier and Rolland \cite{MR}, where it is shown that the homogeneous polynomials $x_i^qx_j - x_ix_j^q$, for $0\le i < j \le m$, generate $\Gs$. We 
supplement this by showing that the above generators themselves form a universal Gr\"obner basis of $\Gs$. 
More importantly, we systematically develop 
a useful notion of projective reduction (for monomials, and more generally, polynomials), that provides canonical representatives for the cosets in $\overline{\FF}_q[x_0, \dots , x_m]/\Gs$. 
A projective analogue of the footprint bound \eqref{FBound} is easily obtained 
using the theory of Hilbert functions. However, unlike in the affine case, the ideal $\I_q:= \I +\Gs$ of $\overline{\FF}_q[x_0, x_1, \dots , x_m]$ need not be a radical ideal 
even if $\I$ is a radical ideal. Nonetheless, we  prove a projective analogue of the $\Fq$-footprint formula \eqref{FFormula} using  a classical result of Macaulay. 
This is then used to derive a \emph{projective $\Fq$-footprint bound} analogous to \eqref{AFBound}, and, as an application, we deduce Serre's inequality  \eqref{SerreIneq} from it. 
In  a forthcoming work \cite{BDG2018}, it will be shown how the basic 
results in this paper can be combined with techniques from extremal combinatorics 
to obtain newer results concerning the determination of $e_r(d,m)$. 

After this work was completed, we became aware of the works of Carvalho, Neumann and Lopez \cite{CNL} and of Gonz\'alez-Sarabia, Mart\'inez-Bernal, Villarreal and Vivares \cite{GMVV}. 
In \cite{CNL}, the so-called projective nested cartesian codes are studied, and they obtain a Gr\"obner basis for the vanishing ideal of the set $X$ of points in $\PP^m(\Fq)$ having homogeneous coordinates in $A_0\times  \dots \times A_m$, where $A_0, \dots , A_m$ are subsets of $\Fq$ satisfying  a certain ``nested condition''. The case $A_0 = \dots = A_m=\Fq$ corresponds to $X=\PP^m(\Fq)$ that we consider here. Moreover, they determine the minimum distance of the corresponding codes, which amounts to a generalization of Serre's inequality shown here as an illustration of projective $\Fq$-footprint bound. However, they restrict to graded lexicographic order, and the methods are different. In \cite{GMVV}, the so-called generalized minimum distance functions are studied and their Lemma 3.4 is related to our projective $\Fq$-footprint formula (Theorem~\ref{thm:PFF} of this paper). Again, the formulation and methods of proof are quite different. 
%
%
\section{Projective Reduction and Fermat Polynomials}

Let $q$ be a prime power and let $m$ be a positive integer (which are kept fixed throughout the paper). 
It is well-known that an algebraic closure $\overline{\FF}_q$ of  the finite field $\Fq$ with $q$ elements is explicitly given by $\cup_{j\ge 1} \FF_{q^j}$ and thus we will assume that algebraic field extensions of $\Fq$ are subfields of this algebraic closure. Let $x_0, x_1, \dots , x_m$ be independent indeterminates over $\overline{\FF}_q$, and let 
$$
\Mon: = \text{the set of all monomials in $x_0, x_1, \dots , x_m$}.
$$
Moreover, for any nonnegative integer $e$, let 
$$
\Mon_e:= \{\mu \in \Mon : \deg \mu = e\}. 
$$
For a monomial $\mu=x_0^{a_0}\cdots x_{m}^{a_{m}}$ in $\Mon$, we let $\supp(\mu) := \{i\in \{0,1, \dots , m\} : a_i > 0\}$. 

\subsection{Projective Reduction} 
\label{subsec:2.1}
The projective analogue of the classical notion of reduction of a polynomial that was alluded to in the introduction is the following. 

\begin{definition}\label{def:monred}
Let $\mu \in \Mon$. If $\mu =1$, then we define $\overline{\mu}:= 1$. If $\mu \ne 1$, then there is a unique $\ell \in \Z$ with 
$0\le \ell  \le m$ such that 
$\mu=x_0^{a_0}\cdots x_{\ell}^{a_{\ell}}$ for some nonnegative integers $a_0, \dots , a_{\ell}$ with $a_{\ell}>0$. 
For $0 \le i \le \ell -1$, let $\overline{a}_{i} \in \{0, 1, \dots,q-1\}$ be the unique integer such that $\overline{a}_{i}=a_{i}$ if $0\le a_{i} \le q-1$, while  $\overline{a}_{i} \equiv a_{i} \pmod{q-1}$ and 
$1\le \overline{a}_{i} \le q-1$ if $a_{i} \ge q$. We then define 
%
$$
\overline{\mu}:=x_0^{\overline{a}_0}\cdots x_{\ell-1}^{\overline{a}_{\ell-1}}x_{\ell}^{a_{\ell}+\sum_{ i=0}^{\ell -1}a_{i}-\overline{a}_{i}}.
$$
It is clear that the monomial $\overline{\mu}$ is uniquely determined by $\mu$.
We let 
$$
\Monb : = \{ \mu \in \Mon :\overline{\mu}=\mu \} 
 \quad \text{and for any $e\ge 0$,} \quad 
\Monb_e:= \{\mu \in \Monb : \deg \mu = e\},
$$
The map $\mu \mapsto \overline{\mu}$ of $\Mon \to \Monb$ extends by linearity to polynomials. More precisely, if $k$ is an algebraic field extension of $\Fq$ and $f \in k[x_0,\dots,x_m]$, then there are unique nonzero scalars $c_1,\dots,c_N \in k$ and distinct monomials $\mu_1,\dots,\mu_N \in \Mon$ such that 
$$
f=c_1\mu_1+\cdots+c_N\mu_N \quad \text{and we define} \quad 
\overline{f}:=c_1\overline{\mu}_1+\cdots+c_N\overline{\mu}_N.
$$
Note that $\overline{f}$  is uniquely determined by $f$. Moreover, using the empty sum convention, if $f$ is the zero polynomial, then so is $\overline{f}$. We call $\overline{f}$ the \emph{projective reduction} of $f$.  We say that $f$ is \emph{projectively reduced} if $\overline{f}=f$. 
Two polynomials in $k[x_0, x_1, \dots , x_m]$ are said to be \emph{projectively equivalent} if they have the same projective reduction. 
\end{definition}

Some elementary properties of projective reduction are summarized below. 

\begin{proposition}\label{prop:basicprojred}
Let $k$ be an algebraic field extension of $\Fq$.  
\begin{enumerate}
\item[{\rm (i)}] The map  $f \mapsto \overline{f}$ of $k[x_0, \dots , x_m] \to k[x_0, \dots , x_m]$ is a $k$-linear homomorphism (of vector spaces) and it maps homogeneous polynomials of degree $e$ to homogeneous polynomials of degree $e$, where $e$ is any nonnegative integer. 
\item[{\rm (ii)}] The map  in {\rm (i)} is idempotent, i.e.,  $\overline{ \overline{f} } =  \overline{f} $ for all  $f\in k[x_0, \dots , x_m]$. 
\item[{\rm (iii)}] $\overline{fg} =  \overline{ \overline{f}\overline{g}} $ for all  $f,g\in k[x_0, \dots , x_m]$. 
\end{enumerate}
\end{proposition}
\begin{proof}
Both (i) and (ii) are obvious from the definition once we note that $\deg \overline{\mu}= \deg \mu$ and 
$\overline{\overline{\mu}} =\overline{\mu}$ for all $\mu\in \Mon$. Moreover, (iii) follows from $k$-linearity once we verify directly that $\overline{\mu \nu}=\overline{\overline{\mu}\,\overline{\nu}}$ for any $\mu,\nu \in \Mon$. 
\end{proof}%

Let us define $\Monb^{(0)}:=\{x_0^{a}: a\ge 0\}$ and for $1\le \ell \le m$, 
$$
\Monb^{(\ell)}:=\{x_0^{a_0}\cdots x_{\ell}^{a_{\ell}} : 0\le a_{i} < q \text{ for } 0\le i < \ell  \text{ and } a_{\ell} >0 \}.
$$
Then for any nonnegative integer $e$, we have the disjoint union decompositions: 
\begin{equation}
\label{eq:disjunion}
\Monb = \coprod_{\ell=0}^m \Monb^{(\ell)}  \quad \text{and} \quad
\Monb_e = \coprod_{\ell=0}^m \Monb_e^{(\ell)}  \quad \text{where} \quad \Monb_e^{(\ell)}:= \Monb^{(\ell)}\cap \Monb_e.
\end{equation}
It is not difficult to obtain a general formula for the cardinality of $\Monb_e$ for any 
$e \ge 0$, but this will not concern us here. For now, we note that if $e\ge m(q-1)+1$, then for $0\le \ell  \le m$, the elements of $\Monb_e^{(\ell)}$ are precisely the monomials of the form $ x_0^{a_0}\cdots x_{\ell-1}^{a_{\ell-1}}x_\ell^{e- a_0 - \cdots - a_{\ell-1}}$, where $a_0 , \dots , a_{\ell-1} $ vary arbitrarily in $\{0, 1, \dots , q-1\}$ 
and the exponent of $x_{\ell}$ is 
$>0$ due to the condition on $e$.  Hence 
$$
|\Monb_e^{(\ell )} | = q^{\ell} \quad \text{and consequently,} \quad | \Monb_e| = \sum_{\ell=0}^m q^{\ell} = \p_m \quad \text{for all } e\ge m(q-1)+1.
$$
%
%

\subsection{Fermat Polynomials and the Division Algorithm} 
Define 
$$
\Phi_{ij} : = x_i^qx_j-x_ix_j^q \quad \text{for } 0\le i < j \le m. 
$$
Following \cite{DG1}, we call these 
the \emph{Fermat polynomials}. We shall denote by $\Phi$ the ordered $\binom{m+1}{2}$-tuple $\left( \Phi_{00}, \Phi_{01}, \dots , \Phi_{0m}, \Phi_{12}, \dots , \Phi_{m-1 , m}\right)$.
Suppose $k$ is an algebraic field extension of $\Fq$. 
The ideal of $k[x_0,    \dots , x_m]$ 
generated by $\{   \Phi_{ij} : 0\le i < j \le m\}$ will be denoted by $\Gamma_q(k)$. Let us fix a monomial order $\preccurlyeq$ on the set $\Mon$ of all monomials in $x_0, \dots , x_m$ in such a way that 
$$
x_0 \succ x_1 \succ \dots \succ x_m.
$$
For $0\ne f\in \overline{\FF}_q[x_0, \dots , x_m]$, we denote by $\lm (f)$ the \emph{leading monomial} of $f$, i.e., the largest  monomial (with respect to $\preccurlyeq$) appearing in $f$ with a nonzero coefficient. For example, $\lm(\Phi_{ij}) = x_i^qx_j$ for $ 0\le i < j \le m$. Given 
any 
$I \subseteq k[x_0,    \dots , x_m]$, we will 
denote by $\LT(I)$ the ideal of $k[x_0,    \dots , x_m]$ generated by 
\hbox{$\{\lm(f) :f\in I \text{ with } f \ne 0\}$.} 
This may be referred to as the \emph{leading term ideal} 
of $I$  (with respect to $\preccurlyeq$).  Let us recall that if $\phi= (\phi_1, \dots , \phi_s)$ is an ordered $s$-tuple of 
nonzero polynomials in $k[x_0,    \dots , x_m]$, then any $f\in k[x_0,    \dots , x_m]$ can be written as $f= f_1\phi_1+ \cdots f_s\phi_s +g$ for some $f_1, \dots , f_s, g\in k[x_0,    \dots , x_m]$ such that no monomial appearing in $g$ with a nonzero coefficient is divisible by any of $\lm(\phi_1), \dots , \lm(\phi_s)$. Here $g$ is uniquely determined by $f$ and the ordered tuple $\phi$. We call $g$  
the \emph{remainder} of $f$ upon division by $\phi$, and we write $f\to_{\phi}g$ to indicate this. 
Note that one arrives at $g$ and $f_1, \dots , f_s$ by a specific \emph{division algorithm} that is guaranteed to terminate in a finite number of steps, and thus we can and will speak of the first step of this algorithm for division of $f$ by $\phi$. 
(cf. \cite[Ch. 2]{CLO}.)  

Here is a connection between Fermat polynomials and projective reduction. 

\begin{proposition}
\label{prop:diveffect}
Let $\mu \in \Mon$. Then 
$\mu \rightarrow_{\Phi} \overline{\mu}.$ Two monomials in $\Mon$ have the same remainder upon  division 
by $\Phi$ if and only if they are projectively equivalent.
\end{proposition}
\begin{proof}
If $\mu=1$, then $\overline{\mu}=1$ and clearly $\mu \to_{\Phi} \overline{\mu}$. 
Suppose  $\mu=x_0^{a_0}\cdots x_{\ell}^{a_{\ell}}$ 
for some nonnegative integers $a_0, \dots , a_{\ell}$ with $a_{\ell}>0$. We define the \emph{weight} of $\mu$ as follows.
$$
\mathrm{wt}(\mu)=\sum_{i=0}^{{\ell} -1} (a_i-\overline{a}_i)q^{m-i}, \quad \text{where $\overline{a}_i$ are 
as in Definition \ref{def:monred}.} 
$$
Note that  
$\mathrm{wt}(\mu) \ge 0$. Moreover, 
$\mathrm{wt}(\mu)=0$ if and only if $\mu$ is projectively reduced.

{\bf Claim 1:} 
$\mu \rightarrow_{\Phi} \nu$ for some projectively reduced monomial $\nu \in \Monb$. 

We prove the claim by induction on $\mathrm{wt}(\mu)$. First, suppose 
$\mathrm{wt}(\mu)=0$, i.e., $\mu = \overline{\mu}$. 
Then $0 \le a_i \le q-1$ for all $i=0,\dots \ell -1$. Since $\lm(\Phi_{i\, j})=x_i^qx_j$ with $i<j$, it follows that 
none of the leading monomials of  $\Phi_{i \, j}$ divides $\mu$. Hence $\mu \rightarrow_{\Phi} \mu = \overline{\mu}$. 

Now suppose $\mathrm{wt}(\mu)>0$, and the claim holds for monomials of weight smaller than 
$\mathrm{wt}(\mu)$. Then 
$\mu$ is not projectively reduced. Let $i$ be the smallest index such that $a_i>q-1$. Note that $i < \ell$, since $\mu$ is not projectively reduced. 
Further let $j$ be the smallest index such that $i < j \le \ell$ and $a_j>0$.
Then the first step of the division algorithm for the division of $\mu$ by $\Phi$  will be the reduction
\begin{equation}\label{eq:stepindivision}
\mu \rightarrow \mu-\dfrac{\mu}{x_i^{q}x_j}\Phi_{i\, j}= \mu' \quad \text{where} \quad \mu':= x_i^{a_i-q+1}x_j^{a_j+q-1}\prod_{\substack{0 \le s \le \ell \\ s \ne i, \; s\ne j}} 
x_s^{a_s}.
\end{equation}
However,
$$
\mathrm{wt}\left( \mu' 
\right)=\mathrm{wt}(\mu)-(q-1)(q^{m-i}-q^{m-j})<\mathrm{wt}(\mu).
$$
Thus Claim 1 follows from the induction hypothesis. 

{\bf Claim 2:} 
If $\mu \rightarrow_{\Phi}\nu$ for some 
$\nu \in \Monb$, then $\nu=\overline{\mu}$.

To prove this claim, note that if $\mu$ is transformed to $\mu'$ by the first step of the division algorithm, then from \eqref{eq:stepindivision}, we see that $\supp(\mu) = \supp(\mu')$ and moreover, 
if  $\mu'=x_0^{a'_0}\cdots x_{\ell}^{a'_{\ell}}$ 
for some nonnegative integers $a'_0, \dots , a'_{\ell}$, then 
\begin{equation}\label{eq:exponents}
a_i \equiv a'_i ({\rm mod} \; q-1) 
 \quad \text{and} \quad  a_i =0 \Longleftrightarrow a'_i = 0 \quad \text{for }0\le i \le \ell.
\end{equation}
Likewise if in the next step of the division algorithm, $\mu'$ is transformed to some $\mu''$ in $\Mon$, then 
  $\supp(\mu') = \supp(\mu'')$ and  the exponents of $\mu''$ will satisfy conditions similar to that in \eqref{eq:exponents}. It follows that if $\mu \rightarrow_{\Phi}\nu$, then $\nu =x_0^{b_0}\cdots x_{\ell}^{b_{\ell}}$ 
for some nonnegative integers $b_0, \dots , b_{\ell}$ that satisfy \eqref{eq:exponents} with $a'_i$ replaced by $b_i$. Since $\nu$ is projectively reduced, this implies that $\nu=\overline{\mu}$, is desired. 

The proposition 
now follows directly from the above two claims.
\end{proof}

\begin{corollary}\label{cor:divonpol}
Let $k$ be an algebraic extension of $\Fq$ 
and let 
$f \in k[x_0,\dots,x_m]$. Then $f \rightarrow_{\Phi} \overline{f}.$ Moreover $\overline{f}=0$ if and only if $f \in \Gs (k)$. 
\end{corollary}

\begin{proof}
Let us first prove that $f \rightarrow_{\Phi} \overline{f}$. 
Write  $f=c_1\mu_1+\cdots+c_N\mu_N$, with $\mu_1,\dots,\mu_N$ distinct monomials in $\Mon$ and $c_1,\dots,c_N$ nonzero elements of $k$. We define 
$$
\mathrm{wt}(f):=\mathrm{wt}(\mu_1)+\cdots+\mathrm{wt}(\mu_N),
$$
where for $\mu\in \Mon$, $\mathrm{wt}(\mu)$ is as in the proof of Proposition~\ref{prop:diveffect}. 
Then $f$ is projectively reduced if and only if $\mathrm{wt}(f)=0$. Hence if $\mathrm{wt}(f)=0$, there is nothing to prove.

Now suppose that $\mathrm{wt}(f)>0$. Equation \eqref{eq:stepindivision} implies that a single step in the division algorithm will replace a monomial occurring in $f$ by a projectively equivalent monomial of lower weight. Hence if $\mathrm{wt}(f)>0$, the first step of the division algorithm will replace $f$ by a polynomial $g$ satisfying $\overline{g}=\overline{f}$ and $\mathrm{wt}(g)<\mathrm{wt}(f)$. By induction, the first assertion in the corollary follows.

We now prove the second assertion. To this end, note that since $f \rightarrow_{\Phi} \overline{f}$, in case $\overline{f} =0$, it follows from the division algorithm that $f$ is in the ideal generated by 
$\{   \Phi_{ij} : 0\le i < j \le m\}$, i.e.,  $f\in \Gamma_q(k)$. Conversely if $f\in \Gamma_q(k)$, then 
we can write 
$f=\sum_{i<j}f_{i\, j}\Phi_{i\, j}$ for some  $f_{i\, j}\in k[x_0,\dots,x_m]$. Using parts (i) and (iii) of 
Proposition~\ref{prop:basicprojred}, we see that
$$
\overline{f}=\sum_{i<j}\overline{f_{i\, j}\Phi_{i\, j}}=\sum_{i<j}\overline{\overline{f_{i\, j}} \, \overline{\Phi_{i\, j}}}=0,
$$ 
where the last equality follows since 
$\overline{\Phi_{i\, j}}= x_i x_j^q - x_j^q x_i = 0$ for 
$0\le i < j \le m$. 
This completes the proof. 
\end{proof}

\subsection{Vanishing Ideal of $\PP^m(\Fq)$} Let $k$ be a field. Recall that the \emph{vanishing ideal} over $k$ of 
a subset  $\X$ of $\PP^m(k)$ is the ideal $I(\X)$ of $k[x_0, x_1, \dots , x_m]$ generated by the homogeneous polynomials in $k[x_0, x_1, \dots , x_m]$ that vanish at every point of $\X$. 
The following result was proved by Mercier and Rolland \cite[Thm. 2.1]{MR}. 

\begin{theorem}
\label{thm:MR}
The vanishing ideal over $\Fq$ of $\, \PP^m(\Fq)$ is $\Gs(\Fq)$, i.e., it is the ideal of $\Fq[x_0,\dots,x_m]$ generated by the Fermat polynomials $x_i^qx_j-x_ix_j^q$  ($0\le i < j \le m$).
\end{theorem}

We observe that this can be used to deduce a more general result. 

\begin{corollary}\label{cor:Fermatgenerate}
Let $k$ be an algebraic extension of $\Fq$. Then the vanishing ideal over $k$ of the set 
$\PP^m(\Fq)$ of all $\Fq$-rational points of $\PP^m(k)$ is $\Gs(k)$. 
\end{corollary}

\begin{proof}
Let $I$ 
denote the vanishing ideal of $\PP^m(\Fq)$ over $k$. 
Clearly, $\Gs(k)\subseteq I$. To prove the other inclusion, suppose $f\in I$. 
Then $f$ is a finite sum of the form 
$$
f = \sum c_{i_0 i_1 \dots i_m}x_0^{i_0}x_1^{i_1} \cdots x_m^{i_m} \quad \text{where the coefficients } c_{i_0 i_1 \dots i_m} \text{ are in } k. 
$$
Since this is a finite sum, there is a positive integer $e$ such that $\FF_{q^e} \subseteq k$ and 
all the coefficients are in $\FF_{q^e}$. Now let $\{\alpha_1, \dots , \alpha_e\}$ be a $\Fq$-basis of $\FF_{q^e}$.
Then 
$$
c_{i_0 i_1 \dots i_m} = \sum_{j=1}^e c_{i_0 i_1 \dots i_m}^{(j)} \alpha_j  \quad \text{for some } c_{i_0 i_1 \dots i_m}^{(j)} \in \Fq.
$$
Consequently, $f = \sum_{j=1}^e f_j \alpha_j$ for some $f_j \in \Fq[x_0, x_1, \dots , x_m]$. Since $f(P)=0$ for all $P\in \PP^m(\Fq)$, we see that $f_j(P)=0$ for all $P\in \PP^m(\Fq)$ and $j=1, \dots , e$. Hence 
$f_j$ is in the  vanishing ideal of $\PP^m(\Fq)$ over $\Fq$, and so by Theorem~\ref{thm:MR}, $f_j \in \Gs(\Fq)$ for each $j=1, \dots , e$. Consequently, $f\in \Gs(k)$. 
%
\end{proof}

\begin{remark} 
{\rm 
An argument similar to the proof of Corollary~\ref{cor:Fermatgenerate} will show that
$\Gs(\Fq) = \Gs(k) \cap \Fq[x_0, \dots , x_m]$. In other words, the vanishing ideal of $\PP^m(\Fq)$ over $\Fq$ is the contraction of the vanishing ideal of $\PP^m(\Fq)$ over $k$ to $\Fq[x_0, \dots , x_m]$.  
An alternative proof of Corollary~\ref{cor:Fermatgenerate} that is independent of Theorem~\ref{thm:MR}, can be found in \S 3.3 of the expository article \cite{G}. 
}
\end{remark} 

Recall that a \emph{universal Gr\"obner basis} of an ideal $I$ in a polynomial ring over a field
is a subset of $I$ 
that is a Gr\"obner basis of $I$ with respect to any term order. 

\begin{theorem}
\label{thm:GBforI}
Let $k$ be an algebraic extension of $\Fq$ and let $I$ be the vanishing ideal of $\PP^m(\Fq)$ over $k$.  Then the set $\{\Phi_{ij} : 0\le i < j \le m\}$ of Fermat polynomials is  a reduced Gr\"obner basis of $I$, and also 
a universal Gr\"obner basis of $I$. 
 \end{theorem}

\begin{proof}
We know from Corollary~\ref{cor:Fermatgenerate} that the Fermat polynomials generate $I$. 
Moreover, for any $i,j,r,s\in \Z$ with $0\le i < j \le m$ and $0\le r< s \le m$, it is clear that the $S$-polynomial $S(\Phi_{ij}, \Phi_{rs})$ is in the ideal $\Gs(k)$. Hence by 
Corollary~\ref{cor:divonpol}, $S(\Phi_{ij}, \Phi_{rs}) \to_{\Phi} 0$. So it follows from  Buchberger's criterion \cite[Ch. 2, \S 6]{CLO} that $\{\Phi_{ij} : 0\le i < j \le m\}$ forms a  Gr\"obner basis  of $I$. 
Further, $\lm(\Phi_{ij}) = x_i^qx_j$ is clearly not divisible by any of the monomials appearing in $\Phi_{rs}$ whenever 
 $i,j,r,s\in \Z$ are such that $0\le i < j \le m$, $0\le r< s \le m$, and $(i,j) \ne (r,s)$. Thus, we see that 
 $\{\Phi_{ij} : 0\le i < j \le m\}$ is a  reduced Gr\"obner basis  of $I$. 
To see that this set is also a universal Gr\"obner basis, it suffices to note that thus far, we had put no condition on the term order $\preccurlyeq$, except that the variables are ordered as $x_0 \succ x_1 \succ \cdots \succ x_m$. Changing the order on the variables amounts to considering the Fermat polynomials $\Phi_{ij}$ with the variables permuted. But any such permutation of variables will 
leave the set $\{\Phi_{ij} : 0\le i < j \le m\}$ invariant, except for 
a sign change in some of the  $\Phi_{ij}$'s. 
Thus we see that 
$\{\Phi_{ij} : 0\le i < j \le m\}$ is a  Gr\"obner basis  of $I$ with respect to any term order. 
\end{proof}

\begin{corollary}
\label{cor:GBLT}
Let $k$ be an algebraic extension of $\Fq$ and let $I$ be the vanishing ideal of $\PP^m(\Fq)$ over $k$.  
Then $\{x_i^qx_j : 0\le i< j\le m\}$ is a Gr\"obner basis of $\LT(I)$. 
\end{corollary}

\begin{proof}
By Theorem~\ref{thm:GBforI}, $\{\Phi_{ij} : 0\le i < j \le m\}$ is a Gr\"obner basis of $I$. 
Also, $\lm(\Phi_{ij}) = x_i^qx_j$ for $0\le i < j \le m$. So $\{x_i^qx_j : 0\le i< j\le m\}$ 
generates $\LT(I)$. Moreover, it is easily seen that $S(x_i^qx_j , \, x_r^qx_s) =0$ for any $i,j,r,s\in \Z$ with $0\le i < j \le m$ and $0\le r< s \le m$. Thus, from Buchberger's criterion, it follows that 
$\{x_i^qx_j : 0\le i< j\le m\}$  is a Gr\"obner basis of $\LT(I)$. 
\end{proof}

We end this section with the following useful consequences of the determination of Gr\"obner bases of the vanishing ideals of projective spaces over finite fields. 

\begin{corollary}
\label{cor:generatorsLT}
Let $k$ be an algebraic extension of $\Fq$ and let $I$ be the vanishing ideal of $\PP^m(\Fq)$ over $k$.  
Then every element of $k[x_0, \dots , x_m]/I$ can be uniquely written as $f+I$, where $f$ is a projectively reduced polynomial. Moreover, every element of $k[x_0, \dots , x_m]/\LT(I)$ can also be uniquely written as $f+\LT(I)$, where $f$ is a projectively reduced polynomial. 
\end{corollary}

\begin{proof}
Let $f \in k[x_0,\dots,x_m]$. By Corollary \ref{cor:divonpol} and  Corollary~\ref{cor:Fermatgenerate}, 
 $f-\overline{f} \in \Gs(k) = I$. Hence $f+I(\PP^m(\Fq))=\overline{f}+I(\PP^m(\Fq)).$ Moreover, if $f,g$ are two projectively reduced polynomials such that $f-g \in I$, then 
 Corollary~\ref{cor:Fermatgenerate}, Corollary \ref{cor:divonpol} and Proposition~\ref{prop:basicprojred} imply that  $f-g=\overline{f-g}=0,$ This proves the first assertion. 

For the second assertion, note that if $\mu \in \Mon$ is not projectively reduced, then $\mu$ is divisible by $x_i^qx_j$ for some $i,j\in \Z$ with $0\le i< j \le m$, and thus $\mu \in \LT(I)$. This shows that 
 every coset in $k[x_0, \dots , x_m]/\LT(I)$  can be written as 
 $f+\mathrm{LT}(I)$ for some
 projectively reduced polynomial $f$. Moreover, if $f,g$ are two distinct projectively reduced polynomials such that $f-g \in \mathrm{LT}(I)$, then by Corollary \ref{cor:GBLT}, 
 $\lm(f-g)$ is divisible by an element of $\{x_i^qx_j \mid 0 \le i < j \le m\}.$  But this is impossible, since $f-g \ne 0$ and $f-g$ is a  linear combination of projectively reduced monomials.
\end{proof}

\section{Projective Footprint Bound}

We let $\N$ denote the set of all nonnegative integers. For a parameter $e$ varying in the set $\N$, we shall write ``for all $e\gg 0$'' to mean ``for all large enough $e$'', that is to say for all $e\in \N$ that are larger than or equal to  some fixed $e_0\in \N$.  

\subsection{Hilbert Functions and Projective Footprints} 
Let $k$ be a field. Note that the polynomial ring  $k[x_0, \dots, x_m]$ is a graded $k$-algebra. 
For any $e\in \N$, its $e^{\rm th}$ graded component $k[x_0, \dots, x_m]_e$ consists of homogeneous polynomials of degree~$e$ (including the zero polynomial), and it 
is a vector space over $k$ of dimension $\binom{m+e}{e}$. If $\J$ is a homogeneous ideal of $k[x_0, \dots, x_m]$, then $\J = \oplus_{e\ge 0} \J_e$, where $\J_e:= \J \cap  k[x_0, \dots, x_m]_e$ for $e\in \N$. The \emph{Hilbert function} 
of $k[x_0, \dots, x_m]/\J$ is the map 
$$
h_{\J}: \N \to \N  \quad  \text{defined by} \quad  h_{\J}(e):=  \dim_k k[x_0, \dots, x_m]_e/\J_e \quad \text{for } e\in \N.
$$
It is well-known that there is a polynomial $\chi_{\J} \in \Q[t]$ 
such that 
$$
h_{\J} (e) = \chi_{\J} (e) \quad \text{for all } e\gg 0.
$$
Clearly $\chi_{\J} $ is uniquely determined by $\J$ and it is called the \emph{Hilbert polynomial} of $\J$. Note that if $\I$ is a homogeneous ideal of $k[x_0, \dots, x_m]$ such that $\J \subseteq \I$, then
\begin{equation}
\label{eq:IJ}
h_{\I}(e) \le h_{\J}(e) \text{ for all } e\in \N \quad \text{and hence} \quad \chi_{\I}(e) \le \chi_{\J}(e) \text{ for all } e\gg 0.
\end{equation}
The following result is classical. A proof can be gleaned, for instance, from \cite[\S 3 of Ch. 8 and 9]{CLO} and \cite[\S 7 of Ch. 1]{H}. 
\begin{theorem} 
\label{thm:ProjBasics}
Let $k$ be an algebraically closed field and let $\J$ be a homogeneous ideal of $k[x_0, \dots, x_m]$, and $\X = V(\J)$ be the corresponding projective variety consisting of all $P\in \PP^m(k)$ such that $F(P)=0$ for all homogeneous $F\in \J$. Then:
\begin{enumerate}
\item[{\rm (i)}] \emph{(Projective Nullstellensatz)} 
If $V(\J)$ is empty, then $\sqrt{\J} \supseteq \langle x_0, x_1, \dots , x_m\rangle$, whereas if 
$\X = V(\J)$ is nonempty, then $I(\X) = \sqrt{\J}$. 
\item[{\rm (ii)}] 
Suppose $\X$ is nonempty and $\I = I(\X)$. Then $\deg \chi_{\I} = \deg \chi_{\J} = \dim \X$, where 
$\dim \X$ denotes the dimension of $\X$ as a projective algebraic variety. 
Moreover if $d = \dim X$ and we write $\chi_{\I}(t) = a_0t^d + \cdots + a_d$, then $\deg \X = d!a_0$, where $\deg \X$ denotes the degree of $\X$, and this is a positive integer with the property 
that if $\I = \langle F\rangle$ for some homogeneous 
$F\in  k[x_0, \dots, x_m]$, then $\deg \X = \deg F$, whereas if $\X$ is a finite set, then $\deg \X = |\X|$.   
\end{enumerate}
\end{theorem}

The next theorem 
is also classical and is sometimes ascribed to Macaulay. See, e.g., 
\cite[\S 3 of Ch. 9]{CLO} for a proof.  Henceforth 
we fix a monomial order $\preccurlyeq$ on the set $\Mon$ of all monomials in $k[x_0, \dots, x_m]$ in such a way that $x_0\succ  \cdots \succ x_m$. The leading monomial $\lm(F)$ of a nonzero $F\in k[x_0, \dots, x_m]$ and the leading term ideal $\LT(J)$ of an ideal $J$ of $k[x_0, \dots, x_m]$ will be understood to be with respect to 
this fixed monomial order. 

\begin{theorem} 
\label{thm:MacaulayThm}
Let $k$ be a field and let $\J$ be a homogeneous ideal of $k[x_0, \dots, x_m]$. 
Then $h_{\J} (e) = h_{\LT(\J)}(e) $ and therefore  $\chi_{\J} (e) = \chi_{\LT(\J)}(e)$  for all $e\in \N$. 
\end{theorem}
%

We define the notion of a projective footprint over an arbitrary field as follows. 

\begin{definition} 
\label{def:PFP}
Let $k$ be a field and let $\J$ be a homogeneous ideal of $k[x_0, \dots, x_m]$. 
Then for any $e\in \N$, the~set
$$
\Delta_e(\J):=\{ \mu \in \Mon_e : \mu \ne \lm(F) \text{ for all homogeneous $F\in \J$ with } F\ne 0 \}. 
$$
is called the \emph{projective footprint} of $\J$ in degree $e$. 
\end{definition}

\begin{theorem}[Projective Footprint Bound]
\label{thm:PFB}
Let $k$ be an algebraically closed field and let $\J$ be a homogeneous ideal of $k[x_0, \dots, x_m]$, such that the corresponding projective variety $\X = V(\J)$ 
is a finite subset of $\PP^m(k)$. Then for all $e\gg 0$, 
\begin{equation}
\label{eq:PFB}
|V(\J)| \le |\Delta_e(\J)| \quad \text{with equality if $\J$ is a radical ideal.}
\end{equation}
\end{theorem}

\begin{proof}
If $\X$ is empty, then the inequality in \eqref{eq:PFB} holds trivially for all $e\in \N$; if, in addition, $\J$ is a radical ideal, then by part (i) of Theorem~\ref{thm:ProjBasics}, 
$\Delta_e(\J)$ is empty 
for all $e\ge 1$ and so the equality holds in this case. Now suppose $\X$ is nonempty. 
Let $\I = I(\X)$ be the vanishing ideal of $\X = V(\J)$ over $k$. Since $\X$ is finite, $\dim \X =0$ and by part (ii) of Theorem~\ref{thm:ProjBasics}, 
both $\chi_{\I}$ and $\chi_{\J}$ are constant polynomials and $\chi_{\I} = |\X|$. On the other hand, since $\J$ is a homogeneous ideal of $k[x_0, \dots , x_m]$, it is easily seen that $\Delta_e(\J) = \{\mu\in \Mon_e : \mu \not\in \LT(\J)\}$.
Now $\LT(\J)$ is a monomial ideal and it is well-known (and easily deduced, e.g., from the division algorithm) that  $\{\mu + \LT(J) : \mu \in \Mon_e , \; \mu \not \in \LT(\J)\}$ is a $k$-basis of 
$\left(k[x_0, \dots , x_m]/\LT(\J)\right)_e$, which implies: 
$$
h_{\LT(\J)}(e) = |\Delta_e(\J)| \quad \text{for all $e\in \N$ and so} \quad 
\chi_{\LT(\J)}(e) = |\Delta_e(\J)| \quad \text{for all }e\gg 0.
$$
Thus in view of \eqref{eq:IJ} and Theorem~\ref{thm:MacaulayThm}, we obtain
$$
|V(\J)|  = \chi_{\I}(e) \le \chi_{\J}(e)  =  \chi_{\LT(\J)}(e) =  |\Delta_e(\J)| \quad \text{for all }e\gg 0,
$$
and the equality clearly holds if $\I=\J$, i.e., 
if $\J$ is a radical ideal. 
\end{proof}

\begin{remark}
\label{rem:PFB}
The above proof of Theorem~\ref{thm:PFB} together with 
Projective Nullstellensatz  shows that if $k$ is algebraically closed and $\J$ a homogeneous ideal of $k[x_0, \dots , x_m]$ such that $V(\J)$ is finite, then $\chi_{\J}$ and $\chi_{\sqrt{\J}}$ are constant polynomials and
$$
\chi_{\sqrt{\J}} = |V(\J)|  \quad \text{and} \quad  \chi_{\J} =  |\Delta_e(\J)| \quad \text{for all }e\gg 0. 
$$
\end{remark}

\subsection{Projective $\Fq$-Footprint Formula} 
Let $k$ 
be an algebraic closure of $\Fq$. By Corollary~\ref{cor:Fermatgenerate}, the vanishing ideal over $k$ of $\PP^m(\Fq)$ is given by 
$$
I \left(\PP^m(\Fq)\right) = \Gs(k) := \langle x_i^q x_j -x_ix_j^q : 0\le i < j \le m\rangle \subseteq k[x_0, \dots , x_m].
$$
If $\J$ is a homogeneous ideal  of $k[x_0, \dots , x_m]$ and if $\X = V(\J)$ is the corresponding projective algebraic variety in $\PP^m(k)$, then we let 
$$
\J_q: = \J + \Gs(k) \quad \text{and} \quad \X_q:= V(\J_q). 
$$
Clearly, $\J_q$ is a homogeneous ideal  of $k[x_0, \dots , x_m]$ and $\X_q(k) =  V (\J_q) = \X(\Fq)$. 

\begin{example}
\label{ex:JqNotRad}
{\rm 
Suppose $m=1$ and $F(x_0, x_1) : = x_0^q x_1 - x_0x_1^q + x_0^{q+1}$. Consider the homogeneous ideal  $\J:= \langle F(x_0, x_1)\rangle$ 
of $k[x_0, x_1]$. 
Observe that $\J = \sqrt{\J}$. To see this, 
note that $k[x_0,x_1]$ is a UFD and $F$ does not have a multiple root in $\PP^1(k)$. Indeed, $F(x_0, x_1) = x_0G(x_0, x_1)$ where $G(x_0, x_1):= x_0^{q-1}x_1 - x_1^q + x_0^q$ does not have $[0:1]$ as a root and also no multiple root of the form $[1:a]$ since the derivative with respect to $x_1$ of $G(1, x_1)$ is never zero. On the other hand, $\J_q = \J + \langle  x_0^q x_1 - x_0x_1^q \rangle$ contains $x_0^{q+1}$, but does not contain $x_0$ (since every nonzero element of $\J_q$ has degree $\ge q+1$). Thus $\J_q$ is not a radical ideal even though $\J$ is a radical ideal. 
}
\end{example}

In light of this example, the following result seems noteworthy. 

\begin{theorem}[Projective $\Fq$-Footprint Formula] 
\label{thm:PFF}
Let $k$ 
be an algebraic closure of $\Fq$. Also, let 
$\J$ be a homogeneous ideal  of $k[x_0, \dots , x_m]$ and $\X = V(\J)$ the corresponding projective algebraic variety in $\PP^m(k)$. Then
$$
|\X(\Fq)|  =  |\Delta_e(\J_q)| \quad \text{for all }e\gg 0.
$$
\end{theorem}

\begin{proof}
We have noted that $\X(\Fq) = \X_q = V(\J_q)$. 
Let $\I:= I(\X_q) = \sqrt{\J_q}$. In view of Remark~\ref{rem:PFB}, both $\chi_{\J_q}$ and $\chi_{\I}$ are constant polynomials and 
$$
\chi_{\I} = |\X(\Fq)|    \quad \text{and} \quad  \chi_{\J_q} =  |\Delta_e(\J_q)| \quad \text{for all }e\gg 0. 
$$
Thus 
it suffices to show that 
$ \chi_{\I} =  \chi_{\J_q}$. To this end, let $\{G_1, \dots , G_r\}$ be a Gr\"obner basis of $\I$ such that each $G_i$ is a homogeneous polynomial in 
$k[x_0, \dots , x_m]$. 
Let $d_i:= \deg G_i$ for $i=1, \dots , r$. Choose a large enough positive integer $s$ such that 
$$
G_i \text{ has all its coefficients in } \FF_{q^s}   \quad \text{and} \quad G_i^{q^s} \in \J_q 
 \quad \text{for each } i=1, \dots , r.
$$ 
Then $G_i^{q^s} (x_0, \dots , x_m) = G_i\big(x_0^{q^s}, \dots , x_m^{q^s}\big)$ for $1\le i \le r$. Hence if 
we let 
$$
H_{ij} := x_j \big( G_i^{q^s} - x_j^{d_i(q^s-1)}G_i \big)  \quad \text{for  } i=1, \dots , r 
\text{ and } j=0,1, \dots , m,
$$ 
then we see that $H_{ij} (P) = 0$ for all $P\in \PP^m(\Fq)$. Indeed, if $P=[a_0: \cdots : a_m]$, 
where $a_0, \dots , a_m \in \Fq$, then clearly $H_{ij} (P) = 0$ when $a_j=0$, whereas in case $a_j=1$, 
$H_{ij}(P) = G_i^{q^s}(a_0, \dots , a_m)- G_i(a_0, \dots , a_m) = G_i(a_0^{q^s}, \dots , a_m^{q^s}) - G_i(a_0, \dots , a_m) =0$ as well. 
Thus by Corollary~\ref{cor:Fermatgenerate}, 
$$
H_{ij} \in \Gs(k)  \quad \text{and so } \quad x_j^{d_i(q^s-1)+1}G_i \in \J_q  \quad \text{for  } i=1, \dots , r 
\text{ and } j=0,1, \dots , m,
$$ 
where the last assertion follows since $\Gs(k) \subseteq \J_q$ and $G_i^{q^s} \in \J_q$ for $1\le i\le r$. 
It follows that $x_j^dG_i\in \J_q$ and $x_j^d\lm(G_i)\in \LT(\J_q)$ for   $1\le i \le r$ and $0\le j\le m$, where 
$$
d:= \max\{d_i(q^s-1)+1 : i=1, \dots , r\}.
$$
Hence 
the monomial ideal 
$\M_q:=\langle x_j^d \lm(G_i): 1\le i \le r, \ 0\le j\le m \rangle$ 
satisfies
\begin{equation}
\label{eq:MqinI}
\M_q \subseteq \LT(\J_q) \subseteq \LT(\I) = \langle \lm(G_i) : 1\le i \le r\rangle,
\end{equation}
where the last equality holds since $\{G_1, \dots , G_r\}$ is a Gr\"obner basis of $\I$. 
Now suppose $e \ge d(m+1) + \max\{d_1, \dots , d_r\}$. We claim that $\LT(\I)_e \subseteq (\M_q)_e$. To see this, let $\mu \in \Mon_e$ be such that $\mu \in LT(\I) = \langle \lm(G_i) : 1\le i \le r\rangle$. Then $\mu = \nu \lm(G_i)$ for some $i\in \{1, \dots , r\}$ and $\nu \in \Mon_{e - d_i}$. Now $e-d_i \ge d(m+1)$, by our choice of $e$, and so $x_j^d \mid \nu$ for some $j\in \{0,1, \dots m\}$. Thus 
$\nu = x_j^d \nu'$ for some $\nu'\in \Mon$, and 
hence  $\mu = \nu' x_j^d \lm(G_i)$. This shows that $\mu \in \M_q$. So the claim is proved. Consequently, 
$h_{\M_q}(e) \le h_{\LT(\I)}(e)$ for all $e \gg 0$. 
This implies that $\chi_{\M_q}$ is a constant polynomial (because $\chi_{\I}$ is a constant polynomial and by Theorem~\ref{thm:MacaulayThm}, $\chi_{\I} =\chi_{\LT(\I)}$), and moreover, $\chi_{\M_q} \le \chi_{\LT(\I)}$. 
On the other hand, by \eqref{eq:IJ} and \eqref{eq:MqinI},  $\chi_{\LT(\I)} \le \chi_{\LT(\J_q)}\le \chi_{\M_q}$. Thus, we obtain $\chi_{\LT(\I)} = \chi_{\LT(\J_q)}$ and so by Theorem~\ref{thm:MacaulayThm}, $\chi_{\I} =\chi_{\J_q}$, as desired. 
\end{proof}

Although very simple, it may be instructive to work out the Hilbert functions in Example~\ref{ex:JqNotRad} so as to illustrate Theorem~\ref{thm:PFF}. 

\begin{example}
\label{ex:JqNotRad2}
{\rm 
Suppose $m=1$ and $F(x_0, x_1)$ as well as $\J$ and $\J_q$ are as in Example~\ref{ex:JqNotRad}. Let $\X := V(\J)$. Evidently, $\X(\Fq) = \{[0:1]\}$ and so $|\X(\Fq)|=1$. Also, $\I:= \sqrt{\J_q} = \langle x_0\rangle$ and thus $\{x_1^e + \I\}$ is a basis of $\left(k[x_0,x_1]/\I\right)_e$; in particular,  $h_{\I}(e) =1$ for all $e\in \N$. On the other hand, 
$$
\J_q = \langle F_1, F_2\rangle \quad \text{where} \quad F_1:= x_0^q x_1 - x_0x_1^q \quad \text{and} \quad F_2:= x_0^{q+1}.
$$
A Gr\"obner basis for $\J_q$ is given by $\{F_1, F_2, F_3, F_4\}$, where $F_1, F_2$ are as above, while $F_3:= x_0^2x_1^q$ and $F_4:= x_0x_1^{2q-1}$. Indeed, 
$F_3$ and $F_4$ are in $\J_q$, and since $F_2, F_3, F_4$ are monomials, the $S$-polynomial $S(F_i, F_j)$ is $0$ for $2\le i< j \le 4$, whereas
$$
S(F_1, F_2) = -F_3, \quad S(F_1, F_3) = -F_4, \quad \text{and} \quad S(F_1, F_4) = - x_0x_1^{3q-2} \in \langle F_4\rangle. 
$$
The Hilbert function of $\J_q$ or equivalently, the cardinality of the footprint of $\J_q$ in each degree, is not difficult to determine by a direct computation, and is given by
$$
h_{\J_q}(e) = \begin{cases} e+1 & \text{if } 0\le e \le q, \\
2q+1 -e & \text{if } q+1 \le e \le 2q-1, \\
1 & \text{if }  e \ge 2q. \end{cases}
$$
So the Hilbert function stabilizes after $e=2q$ 
and $\chi_{\J_q}=1 =|\Delta_e(\J_q)|$ for $e\ge 2q$. 
}
\end{example}

\subsection{Projective $\Fq$-Footprint Bound}
We shall now derive a version of Theorems \ref{thm:PFB} and \ref{thm:PFF} that is useful for estimating the number of $\Fq$-rational points of projective algebraic varieties over an algebraic extension $k$ of $\Fq$. 
We first note 
some elementary properties of the fixed monomial order $\preccurlyeq$ on $\Mon$ and  projective reduction. 

\begin{lemma}
\label{lem:lmFbar}
Let $\preccurlyeq$ be a  monomial order on $\Mon$   for which $x_0\succ \cdots \succ x_m$. Then:
\begin{enumerate}
\item[{\rm (i)}] $\overline{\nu} \preccurlyeq \nu$ for all $\nu \in \Mon$. 
\item[{\rm (ii)}] If $F\in k[x_0, \dots , x_m]$ is 
such that $F\ne 0$ and 
$\lm (F)$ is reduced, then $\overline{F} \ne 0$ and $\lm(F) = \lm (\overline{F})$. 
\end{enumerate}
\end{lemma}

\begin{proof}
(i) 
The case $\nu =1$ is trivial. Suppose $\nu \in \Mon$ with $\nu \ne 1$. Then 
$\nu = x_0^{a_0} \cdots x_{\ell}^{a_{\ell}}$ for some $a_0, \dots , a_{\ell} \in \N$ with $a_{\ell}>0$. Let $\overline{a}_0, \dots , \overline{a}_{\ell - 1}$ be as in Definition~\ref{def:monred}. Now
$$
x_j \succ x_{\ell} \Longrightarrow x_j^{a_j} \succcurlyeq  x_j^{\overline{a}_j} x_{\ell}^{a_j - \overline{a}_j} \quad 
\text{for } 0\le j < \ell \quad \text{and hence} \quad \overline{\nu} \preccurlyeq \nu.
$$

(ii) Let $\, 0\ne F\in k[x_0, \dots , x_m]$ be 
such that $\lm (F)$ is reduced. 
Let $c_1, \dots , c_N$ be nonzero elements of $k$ and $\nu_1, \dots , \nu_N$ be monomials in $\Mon$ such that 
$$
F = c_1\nu_1 + \cdots + c_N \nu_N \quad 
\text{ and } \quad 
\nu_1 \succ \dots \succ \nu_N.
$$
Then  $\lm (\overline{F}) = \nu_1$ and by hypothesis, $\nu_1 = \overline{\nu}_1$. 
Thus, in view of (i) above, 
$$
\overline{F} = c_1\overline{\nu}_1 + \dots + c_N \overline{\nu}_N  \quad \text{and} \quad
\overline{\nu}_1 = \nu_1 \succ \nu_j \succcurlyeq \overline{\nu}_j \text{ for } 2 \le j \le N.
$$
It follows that $\overline{F} \ne 0$ and  $\lm (\overline{F}) = \overline{\nu}_1 = \nu_1 = \lm(F)$. 
\end{proof}

We now define a variant of the projective footprint that is specific to the base field $\Fq$ and uses projectively reduced monomials. 

\begin{definition} 
\label{def:FqPFP}
Let $k$ be an algebraic extension of $\Fq$ and let $S$ be a subset of $k[x_0, \dots, x_m]$. 
Then for any $e\in \N$, the~set
$$
\overline{\Delta}_e(S):=\{ \mu \in \Monb_e : \lm(\overline{F}) \nmid \mu  \text{ for all homogeneous $F\in S$ with } \overline{F}\ne 0 \} 
$$
is called the \emph{projective $\Fq$-footprint} of $S$ in degree $e$. Moreover, its complement 
$$
\overline{\nabla}_e(S):= 
\{ \mu \in \Monb_e : \lm(\overline{F}) \mid \mu  \text{ for some homogeneous $F\in S$ with } \overline{F}\ne 0 \}
$$
is called the \emph{projective $\Fq$-shadow} of $S$ in degree $e$. If $S=\{F_1, \dots , F_r\}$, we often write $ \overline{\Delta}_e(F_1, \dots , F_r)$ and $ \overline{\nabla}_e(F_1, \dots , F_r)$ 
for 
$\overline{\Delta}_e(S)$ and $ \overline{\nabla}_e(S)$, respectively. Note that if $S$ contains no 
homogeneous polynomial $F$ with $\overline{F}\ne 0$, then $\overline{\Delta}_e(S) = \Monb_e$ 
and 
$\overline{\nabla}_e(S) = \emptyset$.
Also, for any $S'\subseteq k[x_0, \dots, x_m]$,
\begin{equation}
\label{eq:SSprime}
S \subseteq S' \Longrightarrow \overline{\Delta}_e(S) \supseteq \overline{\Delta}_e(S') \text{ and }
\overline{\nabla}_e(S) \subseteq \overline{\nabla}_e(S').
\end{equation}
\end{definition}

A relation between projective footprints and projective $\Fq$-footprints is as follows. 

\begin{lemma}
\label{lem:PFandPFqF}
Let $k$ be an algebraic extension of $\Fq$ and let $\J$ be a homogeneous ideal 
of $k[x_0, \dots, x_m]$.  As before, let $\J_q := \J + \Gs(k)$. Then 
$$
\overline{\Delta}_e(\J) = \Delta_e(\J_q) \quad \text{for all } e\in \N.
$$
\end{lemma}

\begin{proof}
Let $e\in \N$ and $\mu \in \overline{\Delta}_e(\J)$. Suppose, if possible, $\mu = \lm(F)$ for some homogeneous $F\in \J_q$. Write $F = G + \gamma$, where $G\in \J$ and $\gamma \in \Gs(k)$. Since $\J$ and $\Gs(k)$  are homogeneous ideals, we may assume without loss of generality that both $G$ and $\gamma$ are homogeneous. By Corollary~\ref{cor:divonpol}, 
$\overline{F} = \overline{G}$. Also since $\mu = \lm(F)$ is reduced, by part (ii) of Lemma~\ref{lem:lmFbar}, $\overline{G}\ne 0$ and $\lm(\overline{G}) = \mu$, which contradicts the assumption that  $\mu \in \overline{\Delta}_e(\J)$. Thus  $\overline{\Delta}_e(\J) \subseteq \Delta_e(\J_q)$. 

To prove the other inclusion, first note that $1\in \J \Leftrightarrow 1 \in \J_q$, and that both 
$\Delta_0(\J_q)$ and $\overline{\Delta}_0(\J)$ are equal to $\{1\}$ or the empty set according as $1\not\in \J$ or $1\in \J$. Thus $\Delta_0(\J_q) = \overline{\Delta}_0(\J)$. Now let $e\in \N$ with $e>0$ and 
$\mu \in \Delta_e(\J_q)$. Then there are  $a_0, \dots , a_{\ell} \in \N$ with $a_{\ell}>0$ such that 
$\mu = x_0^{a_0} \cdots x_{\ell}^{a_{\ell}}$. In case $a_j \ge q-1$ for some $j\in \N$ with $j< \ell$, then 
$$
\lm(\Phi_{j\ell}) = x_j^q x_{\ell} \mid \mu \quad \text{and hence} \quad \mu = \lm(\phi) \text{ for some } \phi\in \J_q.
$$
This shows that $\mu$ is reduced, i.e., $\mu \in \Monb_e$. Further, if $\lm(\overline{F})\mid \mu$ for some homogeneous $F\in \J$ with $\overline{F}\ne 0$, then $\mu = \nu \lm(\overline{F}) = \lm(\nu \overline{F})$ for some $\nu \in \Monb$. But since $F \in \J$ and $F - \overline{F}\in \Gs(k)$, we find $\overline{F}\in \J_q$. Thus we obtain a contradiction to the assumption that  $\mu \in \Delta_e(\J_q)$. 
This proves that $\Delta_e(\J_q) \subseteq \overline{\Delta}_e(\J) $. 
\end{proof}

The result that was 
alluded to at the beginning of this subsection is now an easy consequence of the above lemma and the projective $\Fq$-footprint formula.

\begin{theorem}[Projective $\Fq$-Footprint Bound] 
\label{thm:PFqFB}
Let $k$ be an algebraic closure of $\Fq$ and let $F_1, \dots , F_r$ be any homogeneous polynomials in $k[x_0, \dots, x_m]$. Also, let $\X= V(F_1, \dots , F_r)$ be the corresponding projective variety in  $\PP^m(k)$. 
Then 
\begin{equation}
\label{eq:PFqFB}
|\X(\Fq)| \le \left| \overline{\Delta}_e(F_1, \dots , F_r) \right| \quad \text{for all }e\gg 0.
\end{equation}
\end{theorem}

\begin{proof}
Let $\J$ be the ideal of $k[x_0, \dots, x_m]$ generated by $F_1, \dots , F_r$, and let $\J_q:= \J +\Gs(k)$. Then $\X= V(\J)$. So by Theorem~\ref{thm:PFF}, Lemma~\ref{lem:PFandPFqF} and 
\eqref{eq:SSprime}, we obtain 
$$
|\X(\Fq)| = |\Delta_e(\J_q)|  = | \overline{\Delta}_e(\J) | \le \left| \overline{\Delta}_e(F_1, \dots , F_r) \right|
\quad \text{for all }e\gg 0,
$$
as desired. 
\end{proof}

\begin{remark}
The efficacy of the bound in \eqref{eq:PFqFB} and the ease of computing the projective $\Fq$-footprint
often depends on the choice of the homogeneous polynomials 
$F_1, \dots , F_r$ whose common zeros determine the $\Fq$-rational points of the projective variety $\X$. 
We remark that it is always possible to choose $F_1, \dots , F_r$ 
such that:
\begin{itemize}
\item
$F_i$ is projectively reduced  for each $i=1, \dots , r$. 
\item 
$\deg F_i \le (m+1) (q-1)$ for all $i=1, \dots , r$. 
\item
$F_1, \dots , F_r$ are linearly independent (over $k$). 
\item
$\lm(F_1), \dots , \lm(F_r)$ are distinct. 
\end{itemize}
To see this, first note that since $F_i - \overline{F}_i \in \Gs(k)$ for each $i=1, \dots , r$, the $\Fq$-rational points of $V(F_1, \dots , F_r)$ and $V( \overline{F}_1, \dots ,  \overline{F}_r)$ coincide. Thus, replacing $F_i$ by $ \overline{F}_i $, we may assume that each   of $F_1, \dots , F_r$ is projectively reduced. Next, if some $F_i$ is homogeneous of degree $\ge (m+1)(q-1) + 1$, then every monomial 
$\mu = x_0^{a_0} \cdots x_{m}^{a_{m}}$ appearing in $F_i$, we must have $a_j \ge q$ for some $j$ (depending on $\mu$) with $0\le j \le m$. Replacing the exponent $a_j$ by $a_j - (q-1)$ will result in a homogeneous polynomial $G_i$ of degree $\deg F_i - q+1$ that has exactly the same $\Fq$-rational zeros as $F_i$. Again if $\deg G_i \ge (m+1)(q-1) + 1$, then this process can be continued. 
Thus, we can assume, without loss of generality, that $\deg F_i \le (m+1) (q-1)$ for $i=1, \dots , r$. 
Further, if $F_1, \dots , F_r$ are linearly dependent, then one of them, say $F_r$ is a $k$-linear combination of others, and hence $V(F_1, \dots , F_r) = V(F_1, \dots , F_{r-1})$. Thus it is clear that we may assume $F_1, \dots , F_r$ to be linearly independent. Finally, if two or more of the polynomials have the same leading monomials, then we can use the division algorithm to replace 
$F_1, \dots , F_r$ by $G_1, \dots , G_r$ obtained by suitable linear combinations of $F_1, \dots , F_r$ such that $V(F_1, \dots , F_r) = V(G_1, \dots , G_r)$ and $\lm(G_1), \dots , \lm(G_r)$ are distinct. 
Note that $G_1, \dots , G_r$ continue to possess the properties previously arranged for $F_1, \dots , F_r$, i.e., they are projectively reduced,  linearly independent homogeneous polynomials of degree $\le (m+1) (q-1)$. 
\end{remark}

\begin{remark}
We remark that the projective $\Fq$-footprint bound for the number of $\Fq$-rational points of  a projective variety $\X = V(F_1, \dots , F_r)$ in $\PP^m(k)$ is usually superior 
to the bound obtained from regarding $\X$ as an affine cone, say $\Zs$,  in $\AA^{m+1}(k)$ and using the affine $\Fq$-footprint bound. For example, if $q=5$, and $F_1, F_2\in \FF_5[x_0, x_1, x_2]$ are homogeneous polynomials of degree $5$ such that $\lm(F_1) = x_0^4 x_1$ and $\lm(F_2) = x_0^2x_1^3$, then for the cone $\Zs= Z(F_1, F_2)$ in $\AA^{3}$, the affine $\FF_5$-footprint bound $|\overline{\Delta}(F_1, F_2)|$ works out to be $85$, which implies that the number of 
$\FF_5$-rational points of the projective variety $\X = V(F_1, F_2)$ in $\PP^2$ is bounded above by $(85-1)/(5-1)$, which is $21$. On the other hand, for $e\gg 0$, the projective footprint bound $|\overline{\Delta}_e(F_1, F_2)|$ is 
$20$. 
The details of calculations are left to the reader. 
\end{remark}

\section{Serre's Inequality}
In this section we will show that Serre's inequality, as stated in equation \eqref{SerreIneq}, is an easy consequence of the projective $\Fq$-footprint bound. To this end, we will first compute the cardinality of projective $\Fq$-shadows of all possible leading monomials of nonzero homogeneous polynomial of degree $d\le q$. Note that since $d \le q$ any such polynomial 
is necessarily projectively reduced. 

\begin{lemma}
\label{lem:valueonemu}
Let $d\in \N$ with $1 \le d \le q$ and let $\nu=x_0^{b_0}\cdots x_{\ell}^{b_{\ell}} \in \Mon_d,$ where $0 \le \ell \le m$ and $b_0, \dots, b_{\ell} \in \N$ with $b_{\ell} >0$. Then 
$$ |\nabla_e^q(\nu)|=(q-b_0) \cdots (q-b_{\ell -1})\big( 1+(q-b_{\ell})\p_{m- \ell - 1}\big) \quad \text{for all } e\gg 0.
$$ 
\end{lemma}

\begin{proof}
Note that $\overline{\nu} = \nu$ and thus $\overline{\nabla}_e(\nu) = \{\mu \in  \Monb_e : \nu \mid \mu\}$ for any $e\in \N$. We will use the decomposition \eqref{eq:disjunion}.
First, note that if 
$0\le j < \ell$, then  the variable $x_{\ell}$ does not occur with positive exponent in any of the monomials in $\Monb^{(j)}$. Hence 
$$
\overline{\nabla}_e(\nu) \cap \Monb_e^{(j)} = \emptyset \quad \text{and so} \quad 
|\overline{\nabla}_e(\nu) \cap \Monb_e^{(j)}| = 0, \ \makebox{for $0\le j<\ell$ and any } e\in \N.
$$
Next, we observe that for $e \gg 0$,  a monomial $\mu=x_0^{a_0}\cdots x_{\ell}^{a_{\ell}} \in \Monb_e^ {(\ell)}$
is in  $\overline{\nabla}_e(\nu)$ if and only if $b_i \le a_i \le q-1$ for $i=0, \dots , \ell-1$ because $b_{\ell} \le a_{\ell}$  is automatically satisfied for sufficiently large $e$. Hence
$$
|\overline{\nabla}_e(\nu) \cap \Monb_e^{(\ell)}| = (q-b_0) \cdots (q-b_{\ell -1}) \quad \text{for all } e\gg 0. 
$$
Finally, if $\ell < j \le m$, then $\mu=x_0^{a_0}\cdots x_{j}^{a_{j}} \in \Monb_e^ {(j)}$
is in  $\overline{\nabla}_e(\nu)$ if and only if $b_i \le a_i \le q-1$ for $i=0, \dots , \ell$ and $0\le a_i \le q-1$ for $i=\ell+1, \dots , j-1$. Hence 
$$
|\overline{\nabla}_e(\nu) \cap \Monb_e^{(j)}| = (q-b_0) \cdots (q-b_{\ell })q^{j- \ell -1}, \ \makebox{for $\ell < j \le m$ and any } e\in \N.
$$
Since $\Monb_e$ is the disjoint union of $\Monb_e^{(0)}, \dots \Monb_e^{(m)}$, we conclude that 
$$|\overline{\nabla}_e(\nu)|=(q-b_0) \cdots (q-b_{\ell -1})+(q-b_0) \cdots (q-b_{\ell })(1+q+\cdots+q^{m-\ell-1})$$
and the lemma follows.
\end{proof}
\begin{proposition}[Serre's inequality]
Let $F \in \mathbb{F}_q[x_0,\dots,x_m]$ be a nonzero homogeneous polynomial of degree $d\ \le q$ and let $\X=V(F)$. Then 
$$
|\X (\Fq)| \le dq^{m-1} + \p_{m-2}.
$$
\end{proposition}

\begin{proof}
If $\X$ is contained in a hyperplane, then $|\X(\Fq)| \le p_{m-1} \le d q^{m-1} + p_{m-2}$,  which proves Serre's inequality. In particular, this settles the case $d\le 1$. Therefore we now assume that $2 \le d \le q$ and that $\X$ contains $m+1$ points in general position. After a projective linear transformation, we may moreover assume that $\X$ contains the $m+1$ points $\mathbf{e}_j$ ($0\le j\le m$) whose homogeneous coordinates have $1$ in the $j^{\rm th}$ position and $0$ elsewhere. 
This implies that the monomials of the form $x_j^{d}$ ($0\le j\le m$)  do not occur in $F$. Hence 
there are 
$s, \ell, b_s, \dots , b_{\ell}\in \N$ such that 
$$
\nu :=\lm(F) = x_s^{b_s} \cdots x_{\ell}^{b_{\ell}}, \quad 
0\le s< \ell \le m, \ 0 < b_s < d,  \  0 < b_{\ell} < d \text{ and }\sum_{j=s}^\ell b_j = d.
$$
If $\ell > s+1$, then consider 
$\nu_1 := x_s^{b_s} \cdots x_{\ell -2}^{b_{\ell -2}} x_{\ell -1}^{b_{\ell -1} + b_{\ell}}$ 
and note that by Lemma \ref{lem:valueonemu}, 
\begin{align*}
|\overline{\nabla}_e (\nu)|-|\overline{\nabla}_e (\nu_1)| & = N \big( (q-b_{\ell -1})(1 + (q - b_{\ell})\p_{m-\ell -1})-1- (q - b_{\ell -1} - b_{\ell}) \p_{m-\ell}\big) \notag\\
& = N \left( b_{\ell -1}b_{\ell} \p_{m-\ell -1}+b_{\ell}-1\right)\ge 0 \quad \text{for all } e\gg 0,
\end{align*}
where $N:= (q-b_0) \cdots (q - b_{\ell -2})$ and we have used the fact that $\p_{m-\ell} = q\p_{m-\ell -1}+1$. If $\ell > s+2$, then 
we consider $\nu_2 := x_s^{b_s} \cdots x_{\ell -2}^{b_{\ell -2} + b_{\ell -1} + b_{\ell}}$, and 
in a similar way, we obtain 
$|\overline{\nabla}_e (\nu_1)|-|\overline{\nabla}_e (\nu_2)| \ge 0$ for all $e\gg 0$. Continuing in this way, we see that 
\begin{equation}
\label{eq:ssplu}
|\overline{\nabla}_e (\nu)|-|\overline{\nabla}_e (x_s^{b_s}x_{s+1}^{d-b_s})| \ge 0 \quad \text{for all } e\gg 0.
\end{equation}
Moreover, applying Lemma \ref{lem:valueonemu} once again, we obtain 
\begin{equation}
\label{eq:ssnabla}
|\overline{\nabla}_e(x_s^{b_s}x_{s+1}^{d-b_s})| =q^s(q-b_{s})\big( 1+(q-d+b_s)\p_{m-s-2}\big)
\quad \text{for all } e\gg 0.
\end{equation}
Now by the projective $\Fq$-footprint bound (Theorem~\ref{thm:PFqFB}) and the observation in \S\,\ref{subsec:2.1}  that 
$|\Monb_e| = \p_m$ for $e\gg 0$, Serre's inequality \eqref{SerreIneq} is proved if we show that 
\begin{equation}
\label{eq:serreshadow}
|\overline{\nabla}_e (\nu)| \ge \p_m - \left( dq^{m-1} + \p_{m-2}   \right) = (q-d+1)q^{m-1} \quad \text{for all } e\gg 0,
\end{equation}
Thus,  by \eqref{eq:ssplu}, \eqref{eq:ssnabla}, and \eqref{eq:serreshadow}, it suffices to show that the difference
$$
q^s(q-b_{s})\big( 1+(q-d+b_s)\p_{m-s-2}\big)   - (q-d+1)q^{m-1}
$$
is nonnegative. But an easy calculation shows that this difference is equal to 
$$
q^{m-s-1} (a_s -1) (d - a_s -1) + (q - a_s)(d - a_s -1),
$$
which is clearly nonnegative, since $0< a_s < d \le q$. 
\end{proof}

Since example of homogeneous polynomials of degree $d\le q$ attaining the Serre bound \eqref{SerreIneq} are easy to construct (e.g., 
$F(x_0, \dots , x_m):= (x_1 - a_1x_0) \cdots (x_1 - a_dx_0)$, where $a_1, \dots, a_d$ are distinct elements of $\Fq$), the above proposition shows that $e_1(d,m) = dq^{m-1} + \p_{m-2}$ for $d\le q$. It would be interesting if one can also determine $e_r(d,m)$ explicitly for every $r > 1$ and $d\le q$. While this is open, in general, we refer to \cite{BDG2018} and the references therein for some recent results in this direction. 

\section{Acknowledgements}

The authors would like to gratefully acknowledge the following foundations and institutions:
for received support.
Peter Beelen was supported by The Danish Council for Independent Research (Grant No. DFF--4002-00367). Mrinmoy Datta is supported by The Danish Council for Independent Research (Grant No. DFF--6108-00362). Sudhir Ghorpade is partially supported by IRCC Award grant 12IRAWD009
from IIT Bombay. 
Also, 
Peter Beelen would like to thank IIT Bombay where large parts of this work were carried out when he was there in January 2017 as a Visiting Professor. 
Sudhir Ghorpade would like to thank the Technical University of Denmark for a visit of 11 days in June 2017 when some of this work was done. 

\end{document}